\documentclass[11pt,reqno]{amsart}

\usepackage[letterpaper,margin=1.12in]{geometry}
\usepackage{amsmath,amssymb,mathtools}
\usepackage{microtype}
\usepackage{enumitem}
\usepackage{xcolor}
\usepackage[colorlinks=true,linkcolor=blue!55!black,citecolor=blue!55!black,
  urlcolor=blue!55!black]{hyperref}

\allowdisplaybreaks
\numberwithin{equation}{section}

\makeatletter
\def\section{\@startsection{section}{1}%
  \z@{1.3\linespacing\@plus\linespacing}{.5\linespacing}%
  {\normalfont\scshape\centering}}
\makeatother

\newtheorem{theorem}{Theorem}[section]
\newtheorem{proposition}[theorem]{Proposition}
\newtheorem{lemma}[theorem]{Lemma}
\newtheorem{corollary}[theorem]{Corollary}
\theoremstyle{definition}

\theoremstyle{remark}
\newtheorem{remark}[theorem]{Remark}
\newtheorem*{unnumberedremark}{Remark}

\newcommand{\R}{\mathbb R}
\newcommand{\Sph}{\mathbb S}
\newcommand{\Euc}{\mathrm{Euc}}
\newcommand{\Ric}{\operatorname{Ric}}
\newcommand{\Scal}{\operatorname{Scal}}
\newcommand{\Rm}{\operatorname{Rm}}
\newcommand{\Hess}{\operatorname{Hess}}
\newcommand{\Vol}{\operatorname{Vol}}
\newcommand{\AVR}{\operatorname{AVR}}
\newcommand{\Id}{\operatorname{Id}}
\newcommand{\tr}{\operatorname{tr}}
\newcommand{\dd}{\,\mathrm d}
\newcommand{\cH}{\mathcal H}
\newcommand{\cD}{\mathcal D}
\newcommand{\thmref}[1]{\hyperref[#1]{Theorem~\ref*{#1}}}
\newcommand{\propref}[1]{\hyperref[#1]{Proposition~\ref*{#1}}}
\newcommand{\lemref}[1]{\hyperref[#1]{Lemma~\ref*{#1}}}
\newcommand{\corref}[1]{\hyperref[#1]{Corollary~\ref*{#1}}}
\newcommand{\remref}[1]{\hyperref[#1]{Remark~\ref*{#1}}}

\title{pole regularity of Green function and rigidity}

\author{Zhelei Huang}
\address{Capital Normal University\ ,BEIJING,CHINA}
\email{2260501005@cnu.edu.cn}

\author{Jianshen Xiong}
\address{Capital Normal University\ ,BEIJING,CHINA}
\email{2240502158@cnu.edu.cn}


\begin{document}

\begin{abstract}
Let $G(p,\cdot)$ be the normalized minimal positive Green function on a
complete nonparabolic Riemannian manifold with nonnegative Ricci curvature, and set $b=G^{1/(2-n)}$.
Colding established the following rigidity result in dimension three: if
$|\nabla b|^2$ admits a $C^2$ regularity across the pole, then the manifold
must be isometric to Euclidean space. We show that rigidity conclusions
corresponding to the pole regularity condition alone do not hold in higher
dimensions: for any $n\geq4$, we can construct explicit complete nonflat
rotationally symmetric manifolds with nonnegative sectional curvature and
Euclidean volume growth, such that $|\nabla b|^2$ extends smoothly across
the pole.

For any odd integer $n\geq3$, we impose finite-order pointwise curvature
conditions at the pole (where this condition is vacuous when $n=3$), so
that the Green function satisfies the asymptotic expansion
\[
G(p,x)=r^{2-n}+H_p+O_1(r).
\]
Under this curvature assumption, if $|\nabla b|^2$ extends to a
$C^{n-2}$ function near the pole, then the manifold must be isometric to
$\R^n$, and this regularity threshold is also sharp. As special cases: in
dimension three, if $|\nabla b|^2\in C^1$, then the manifold must be
isometric to Euclidean space; in dimension five, it suffices to assume that
the scalar curvature at the pole satisfies $\Scal(p)=0$ and
$|\nabla b|^2\in C^3$. In contrast, for any even dimension $n\geq4$,
there exist complete nonflat examples that are Euclidean in a neighborhood
of the pole and admit a smooth extension of $|\nabla b|^2$ across the pole.
\end{abstract}

\maketitle
\begingroup
\footnotesize
\setcounter{tocdepth}{1}
\tableofcontents
\endgroup
\newpage

\section{Introduction}

Let \((M^n,g)\) be a smooth, complete, connected, noncompact Riemannian
manifold, $n\geq3$, and suppose that $M$ is nonparabolic (that is, the
minimal positive Green function exists; see \cite{Varopoulos}).  We normalize
the minimal positive Green function with pole $p$ by
\begin{equation}\label{eq:green-normalization}
 -\Delta G(p,\cdot)=(n-2)|\Sph^{n-1}|\,\delta_p,
 \qquad
 G(p,x)=r(x)^{2-n}(1+o(1)),
\end{equation}
where \(\Delta=\operatorname{div}\nabla\) and $r(x)=d(p,x)$.  Thus
\(G(p,x)=|x-p|^{2-n}\) on Euclidean space.  Following Colding
\cite{Colding}, set
\begin{equation}\label{eq:b-def}
 b=G^{1/(2-n)}=G^{-1/(n-2)}.
\end{equation}
The function $b$ is a Green-function regularization of the distance from
$p$; on $M\setminus\{p\}$ it satisfies
\begin{equation}\label{eq:Delta-b2}
 \Delta b^2=2n|\nabla b|^2.
\end{equation}

When $\mathrm{Ric} \ge 0,$
Colding proved that, in dimension three,
if \(|\nabla b|^2\) extends as a $C^2$ function across the pole, then the manifold is
Euclidean \cite{Colding}.  It is therefore natural to ask
whether a finite amount of pole regularity of \(|\nabla b|^2\) yields a
similar rigidity theorem in higher dimensions.

Our first theorem gives a negative answer if no curvature condition is
imposed at the pole.

\begin{theorem}\label{thm:smooth-counterexamples}
For every $n\geq4$ and every $a\in(0,1)$, there is a complete,
rotationally symmetric metric $g_a$ on $\R^n$ such that
\[
 \sec_{g_a}\geq0,
 \qquad (\R^n,g_a)\text{ is nonflat},
 \qquad \AVR(M) = \lim_{r\to\infty} \Vol(B_p(r))/(\omega_n r^n)=a^{n-2}>0,
\]
and the normalized minimal Green distance satisfies
\begin{equation}\label{eq:gradb-explicit}
 |\nabla b|^2=\left(\frac{1+ar^2}{1+r^2}\right)^2.
\end{equation}
In particular, \(\lvert\nabla b\rvert^2\) extends smoothly across the pole.
Moreover,
\begin{equation}\label{eq:scalar-at-pole-explicit}
 \Scal_{g_a}(0)=2n(n-4)(1-a).
\end{equation}
Thus \(\Scal_{g_a}(0)>0\) whenever \(n\geq5\).  For \(n=4\), this
particular family instead satisfies \(\Scal_{g_a}(0)=0\), although the
metric is still nonflat.
\end{theorem}

The construction is inverse: we prescribe $b'$ and then recover the warping
function.  The counterexamples show that regularity of \(|\nabla b|^2\)
alone does not directly imply global rigidity and must be combined with
curvature assumptions.  General Green-kernel asymptotics can be developed by
parametrix or isolated-singularity methods; see
\cite{Aubin,GilbargSerrin,Grigoryan}.  We then use Colding's third
monotonicity formula \eqref{eq:third-monotonicity} to obtain rigidity.

For a function $E$ on a punctured neighborhood of $p$, define
\[
 E=O_1(r^k)
 \quad\Longleftrightarrow\quad
 |E|\leq Cr^k\ \text{ and }\ |\nabla E|\leq Cr^{k-1}\quad\text{on}\quad
 M \setminus \{p\}.
\]
In particular, the derivative estimate is part of the notation; it is not
inferred by formally differentiating a bare $O(r^k)$ estimate.

\begin{theorem}
\label{thm:pure-pole}
Let $n\geq3$ be odd, let $p$ be a point of a smooth Riemannian manifold,
and assume \(\Ric\geq0\) near $p$.  If
\begin{equation}\label{eq:curvature-jet-conditions}
\begin{aligned}
 \nabla^{2i}\Scal(p)&=0
 &&(0\leq 2i\leq n-5),\\
 \left.\nabla^j\Rm\right|_p&=0
 &&\left(0\leq j\leq\frac{n-7}{2}\right),
\end{aligned}
\end{equation}
where the conditions are empty when $n=3$, and second condition is empty when $n=5$, then every Green function
normalized by \eqref{eq:green-normalization} has
an expansion
\begin{equation}\label{eq:pure-pole-expansion}
 G(p,x)=r^{2-n}+H_p+O_1(r)
\end{equation}
for a constant $H_p$.  In particular, $H_p\geq0$ when $\mathrm{Ric} \ge 0$ and
 $G(p,\cdot)$ is
the normalized minimal positive Green function.
\end{theorem}

The proof is local.  If $S_\theta=r^{-1}\Id+B_\theta$ is the shape operator of a small
geodesic sphere, the trace Riccati identity gives
\[
 \tr B(r,\theta)
 =-r^{-2}\int_0^r t^2
 \bigl(|B(t,\theta)|^2+\Ric(\partial_r,\partial_r)\bigr)\dd t.
\]
The curvature conditions in \eqref{eq:curvature-jet-conditions},
together with nonnegative Ricci curvature, imply the radial estimates needed
in the Jacobi equation and hence
\[
 \Delta r=\frac{n-1}{r}+a(\theta)r^{n-2}+O(r^{n-1}),
\]
where $a$ is smooth and even on \(\Sph^{n-1}\).  Consequently
\(\Delta r^{2-n}\) has a critical $r^{-1}$ error.  The correcting term is
$r\phi(\theta)$.  Its spherical equation is solvable because the kernel
at this homogeneity consists of the odd degree-one spherical harmonics,
whereas the source $a$ is even.

Combining the local theorem with Colding's monotonicity formula
\eqref{eq:third-monotonicity} gives the positive rigidity statement.

\begin{theorem}\label{thm:odd-rigidity}
Let $n\geq3$ be odd and let \((M^n,g)\) be complete and nonparabolic with
\(\Ric\geq0\).  Let $G(p,\cdot)$ be the normalized minimal positive Green
function and set $b=G^{1/(2-n)}$.  Assume
\eqref{eq:curvature-jet-conditions} at $p$.  If
\begin{equation}\label{eq:gradb-regularity}
 |\nabla b|^2\text{ extends to a }C^{n-2}\text{ function near }p,
\end{equation}
then \((M^n,g)\) is isometric to \((\R^n,g_{\Euc})\).
\end{theorem}

The role of the regularity hypothesis is seen from
\eqref{eq:pure-pole-expansion},
\begin{equation}\label{eq:gradb-intro-expansion}
 |\nabla b|^2
 =1-\frac{2(n-1)}{n-2}H_pr^{n-2}+O(r^{n-1}).
\end{equation}
For odd $n$, the exponent $n-2$ is odd, and
\eqref{eq:gradb-regularity} forces $H_p=0$, and therefore
\[
 |\nabla b|^2=1+O(r^{n-1}).
\]
This asymptotic expansion yields the vanishing of the boundary term in
Colding's third monotonicity formula \eqref{eq:third-monotonicity}, which in
turn leads to the rigidity conclusion.

\begin{unnumberedremark}
The regularity threshold is sharp, as
\corref{cor:odd-sharpness} shows: even when the metric is Euclidean in a
neighborhood of the pole, for odd $n\geq3$ the weaker condition
\(
 |\nabla b|^2\in C^{n-3,1}
\)
no longer guarantees Euclidean rigidity.
\end{unnumberedremark}

Let \((M^n,g)\) be a complete Riemannian manifold, nonparabolic, with
\(\Ric\geq0\).
The low-dimensional consequences include
\corref{cor:dimension-three}: \(|\nabla b|^2\in C^1\) implies
  \((M^3,g)\cong\R^3\);
\corref{cor:dimension-five}: \(\Scal(p)=0\) and
  \(|\nabla b|^2\in C^3\) imply
  \((M^5,g)\cong\R^5\); and
\corref{cor:dimension-seven}: \(\Rm(p)=0\),
  \(\Delta\Scal(p)=0\), and
  \(|\nabla b|^2\in C^5\) imply \((M^7,g)\cong\R^7\).

The paper is organized as follows.  Section~\ref{sec:colding} records the
part of Colding's theory used later.  Section~\ref{sec:counterexamples}
proves \thmref{thm:smooth-counterexamples}.  Section~\ref{sec:jacobi}
proves \thmref{thm:pure-pole}. Section~\ref{sec:global} proves \thmref{thm:odd-rigidity}, while Section~\ref{sec:even} establishes the sharpness of the regularity condition and presents even-dimensional obstructions.

\section{Green functions and Colding's formulas}\label{sec:colding}

We use the sign convention
\[
 \Delta_{\Sph^{n-1}}Y_k=-k(k+n-2)Y_k
\]
for a spherical harmonic of degree $k$.  From $G=b^{2-n}$ and
\(\Delta G=0\) on $M\setminus\{p\}$, one obtains
\begin{equation}\label{eq:Delta-b2-section2}
 \Delta b^2=2n|\nabla b|^2.
\end{equation}

For a function $v\in C^1(M\setminus\{p\})$, Colding
\cite{Colding} considered
\begin{equation}\label{eq:Iv}
 I_v(s)=s^{1-n}\int_{\{b=s\}}v|\nabla b|\dd\sigma,
\end{equation}
and showed
\begin{equation}\label{eq:Iv-prime}
 I_v'(s)=s^{1-n}\int_{\{b=s\}}v_n\dd\sigma.
\end{equation}
Taking $v\equiv1$ and using the Euclidean pole limit
gives
\begin{equation}\label{eq:I1}
 I_1(s)=|\Sph^{n-1}|.
\end{equation}

Colding's area functional is
\begin{equation}\label{eq:A-def}
 A(s)=s^{1-n}\int_{\{b=s\}}|\nabla b|^3\dd\sigma
 =I_{|\nabla b|^2}(s).
\end{equation}
Set
\begin{equation}\label{eq:defect-density}
 \cD=
 \left|\Hess b^2-\frac{\Delta b^2}{n}g\right|^2
 +\Ric(\nabla b^2,\nabla b^2).
\end{equation}

We will use the following results from \cite{Colding}.

\begin{proposition}[Colding]\label{prop:Colding}
Let \((M^n,g)\) be a complete nonparabolic Riemannian manifold with
$n\geq3$ and nonnegative Ricci curvature $\Ric\geq0$. Then:
\begin{enumerate}[label=\textup{$\roman*$}]
\item for \(r>0\), if
  \(\cD\equiv0\) on
  \[
    \Omega_r\setminus\{p\},
    \qquad \Omega_r:=\{x\in M:b(x)\leq r\},
  \]
  then the closed set \(\Omega_r=\{x\in M:b(x)\leq r\}\), with its
  intrinsic metric, is isometric to the closed Euclidean ball
  \(\overline B_r(0^n)\subset\R^n\).  In particular, if this vanishing holds for
  arbitrarily large \(r\), then \(M\) is isometric to \(\R^n\);
\item $A$ is nonincreasing;
\item for $0<s_1<s_2$,
\begin{equation}\label{eq:third-monotonicity}
 s_2^{3-n}A'(s_2)-s_1^{3-n}A'(s_1)
 =\frac12\int_{\{s_1\leq b\leq s_2\}}\cD\,b^{2-2n}\dd\mu;
\end{equation}
\item $|\nabla b|\leq1$;
\item as $x\to p$ and $s\downarrow0$,
\begin{equation}\label{eq:pole-limits}
 \frac{b(x)}{r(x)}\longrightarrow1,
 \qquad |\nabla b|^2(x)\longrightarrow1,
 \qquad A(s)\longrightarrow|\Sph^{n-1}|.
\end{equation}
\end{enumerate}
\end{proposition}

\begin{lemma}\label{lem:zero-defect}
In the setting of \propref{prop:Colding}, suppose that $s^{3-n}A'(s)$
satisfies
\begin{equation}\label{eq:zero-small-defect}
 \lim_{s\downarrow0}s^{3-n}A'(s)=0.
\end{equation}
Then \((M^n,g)\) is isometric to Euclidean space.
\end{lemma}

\begin{proof}
By \eqref{eq:third-monotonicity},
\[
 s_2^{3-n}A'(s_2)-s_1^{3-n}A'(s_1)
 =\frac12\int_{\{s_1\leq b\leq s_2\}}\cD\,b^{2-2n}\dd\mu.
\]
Letting $s_1\to0$ gives
\[
 s_2^{3-n}A'(s_2)
 =\frac12\int_{\{b\leq s_2\}}\cD\,b^{2-2n}\dd\mu,
\]
and hence
\[
 A'(s_2)
 =\frac12s_2^{n-3}\int_{\{b\leq s_2\}}\cD\,b^{2-2n}\dd\mu\geq0.
\]
Thus $A(s)$ is nondecreasing.  By \propref{prop:Colding}(ii), $A(s)$ is
also nonincreasing.  Hence $A'(s)=0$.  By
\eqref{eq:third-monotonicity}, \(\cD=0\) on
\(\{x\in M:b\leq s_2\}\).  Since $s_2$ is arbitrary,
\propref{prop:Colding}(i) gives \(M\cong\R^n\).
\end{proof}

\section{Smooth pole regularity does not imply rigidity}
\label{sec:counterexamples}

Consider a model metric
\begin{equation}\label{eq:model-metric}
 g=\dd r^2+f(r)^2g_{\Sph^{n-1}},\qquad 0\leq r<\infty,
\end{equation}
with $f$ extends smoothly and oddly across $0$, and $f'(0)=1$.  Its radial and tangential sectional
curvatures are
\begin{equation}\label{eq:model-curvatures}
 K_{\mathrm{rad}}=-\frac{f''}{f},
 \qquad
 K_{\mathrm{tan}}=\frac{1-(f')^2}{f^2}.
\end{equation}
The normalized radial minimal positive Green function, when the integral
converges, is
\begin{equation}\label{eq:model-green}
 G(r)=(n-2)\int_r^\infty f(s)^{1-n}\dd s.
\end{equation}
If $b=G^{-1/(n-2)}$, differentiation gives
\begin{equation}\label{eq:b-prime-model}
 b'=\left(\frac{b}{f}\right)^{n-1}.
\end{equation}

\begin{lemma}\label{lem:inverse-construction}
Let $u:[0,\infty)\to(0,\infty)$ be smooth and even near zero, with
$u(0)=1$.  Set
\begin{equation}\label{eq:inverse-def}
 b(r)=\int_0^r u(s)\dd s,
 \qquad
 f(r)=b(r)u(r)^{-1/(n-1)}.
\end{equation}
Assume that $f$ is positive for $r>0$, extends smoothly and oddly across
zero, satisfies $f''\leq0$, and has $f'>0$.  If $b(r)\to\infty$, then
\eqref{eq:model-metric} is complete with \(\sec\geq0\), the manifold is
nonparabolic, its normalized minimal positive Green function is
\(G=b^{2-n}\), and
\begin{equation}\label{eq:gradb-inverse}
 |\nabla b|^2=u^2.
\end{equation}
\end{lemma}

\begin{proof}
Since $u>0$, the function $b$ is increasing and $f>0$ for $r>0$.
The identity $f=bu^{-1/(n-1)}$ is equivalent to
\eqref{eq:b-prime-model}.  Therefore
\[
 \frac{\dd}{\dd r}b^{2-n}=-(n-2)f^{1-n}.
\]
Because $b\to\infty$, integration from $r$ to infinity gives
\eqref{eq:model-green}.  The normalization follows from
\(b(r)=r+O(r^3)\) near zero.

For completeness, we verify minimality rather than only harmonicity.  Let
$G_R$ be the Dirichlet Green function with pole at the origin in the model
ball $B_R$.  Rotational invariance and uniqueness give
\[
 G_R(r)=(n-2)\int_r^R f(s)^{1-n}\dd s.
\]
Thus $G_R\uparrow b^{2-n}$ as $R\to\infty$.  By the monotone-exhaustion
characterization, $b^{2-n}$ is the minimal positive Green function; in
particular, the manifold is nonparabolic.

Concavity and $f'(0)=1$ give $0<f'\leq1$.  Both expressions in
\eqref{eq:model-curvatures} are therefore nonnegative.  Completeness follows
because the radial coordinate has infinite length.  Finally, $b$ is radial
and \(|\nabla b|^2=(b')^2=u^2\).
\end{proof}

Fix $a\in(0,1)$, write $m=n-1$, and define
\begin{equation}\label{eq:explicit-family}
 u_a(r)=\frac{1+ar^2}{1+r^2},
 \qquad
 b_a(r)=ar+(1-a)\arctan r,
 \qquad
 f_a(r)=b_a(r)u_a(r)^{-1/m}.
\end{equation}
Then $b_a'=u_a$.

\begin{proposition}
\label{prop:concavity}
If $m\geq3$, then $f_a'>0$ and $f_a''\leq0$ on \([0,\infty)\).
Moreover, $f_a''<0$ somewhere.
\end{proposition}

\begin{proof}
Suppress the subscript $a$, and put
\[
 c=1-a,
 \qquad D=1+r^2,
 \qquad U=1+ar^2,
 \qquad u=\frac UD.
\]
Since $u'<0$ for $r>0$,
\begin{equation}\label{eq:f-prime}
 f'=u^{(m-1)/m}-\frac bm u^{-(m+1)/m}u'>0.
\end{equation}
Direct differentiation gives
\begin{equation}\label{eq:f-second}
 f''=\frac{u^{-1/m-2}}m
 \left((m-2)u^2u'
 +b\left(\frac{m+1}{m}(u')^2-uu''\right)\right).
\end{equation}
Using
\[
 u'=-\frac{2cr}{D^2},
 \qquad
 u''=-\frac{2c(1-3r^2)}{D^3},
\]
the bracket in \eqref{eq:f-second} equals $2cD^{-4}F(r)$, where
\begin{align}
 F(r)&=-(m-2)rU^2+bK(r),\label{eq:F-def}\\
 K(r)&=U(1-3r^2)+\frac{2(m+1)}mcr^2.\label{eq:K-def}
\end{align}
If $K(r)\leq0$, then $F(r)\leq0$.  If $K(r)>0$, the bounds
\(0<u\leq1\) give $b(r)=\int_0^ru\leq r$, and hence
\[
 F(r)\leq r\bigl(-(m-2)U^2+K(r)\bigr).
\]
Finally,
\begin{align*}
 (m-2)U^2-K(r)
 &=(m-3)
 +\left(1-\frac2m+a\left(2m-3+\frac2m\right)\right)r^2\\
 &\quad+\bigl((m-2)a^2+3a\bigr)r^4.
\end{align*}
All coefficients are nonnegative for $m\geq3$, and the expression is
positive for $r>0$.  Thus $F\leq0$ and $f''\leq0$.  Since the displayed
expression is not identically zero, neither is $f''$.
\end{proof}

\begin{proof}[Proof of \thmref{thm:smooth-counterexamples}]
The function $u_a$ is a smooth positive function of $r^2$, $b_a$ is
smooth and odd, and $f_a=b_au_a^{-1/(n-1)}$ is smooth and odd with
\(f_a'(0)=1\).  \propref{prop:concavity} and
\lemref{lem:inverse-construction} show that the resulting metric is
complete and has nonnegative sectional curvature.  It is nonflat because
\(f_a''\) is not identically zero.

The normalized minimal Green function is $G_a=b_a^{2-n}$, so
\[
 |\nabla b|^2=(b_a')^2
 =\left(\frac{1+ar^2}{1+r^2}\right)^2.
\]
This is smooth in Cartesian coordinates at the origin.

We next compute the scalar curvature at the pole.  Put \(m=n-1\) and
\(c=1-a\).  As \(r\to0\),
\[
 u_a(r)=1-cr^2+O(r^4),
 \qquad
 b_a(r)=r-\frac{c}{3}r^3+O(r^5),
\]
and hence
\begin{align*}
 f_a(r)
 &=b_a(r)u_a(r)^{-1/m}\\
 &=r-\frac{c(m-3)}{3m}r^3+O(r^5)\\
 &=r-\frac{(1-a)(n-4)}{3(n-1)}r^3+O(r^5).
\end{align*}
Consequently,
\[
 f_a'(r)
 =1-\frac{(1-a)(n-4)}{n-1}r^2+O(r^4),
 \qquad
 f_a''(r)
 =-\frac{2(1-a)(n-4)}{n-1}r+O(r^3).
\]
Using \eqref{eq:model-curvatures}, we obtain
\[
 K_{\mathrm{rad}}(0)=K_{\mathrm{tan}}(0)
 =\frac{2(1-a)(n-4)}{n-1}.
\]
All sectional curvatures at the rotationally symmetric pole have this
common value, and therefore
\[
 \Scal_{g_a}(0)
 =n(n-1)K_{g_a}(0)
 =2n(n-4)(1-a).
\]
This is strictly positive for \(n\geq5\), while it vanishes for \(n=4\).
The latter case remains nonflat because \(f_a''\) is not identically zero.

Finally,
\[
 f_a(r)=a^{(n-2)/(n-1)}r+O(1)\qquad(r\to\infty).
\]
Consequently,
\[
 \AVR(g_a)=\lim_{R\to\infty}
 \frac{\Vol_{g_a}(B_R)}{\omega_nR^n}
 =\left(a^{(n-2)/(n-1)}\right)^{n-1}=a^{n-2}>0.
\]
\end{proof}

\section{Proof of Theorem 1.2}\label{sec:jacobi}

Fix $p\in M$, where $p$ is the pole of the Green function.  All estimates
below are uniform in \(\theta\in\Sph^{n-1}\subset T_pM\).  We work for
sufficiently small $t>0$, inside a normal neighborhood of $p$, and use the
curvature convention for which a Jacobi field $Y$ satisfies
\[
 \nabla_{\dot\gamma}^2Y+\Rm(Y,\dot\gamma)\dot\gamma=0.
\]

\subsection{Radial curvature in a fixed vector space}

For \(\theta\in\Sph^{n-1}\subset T_pM\), let
\begin{equation}\label{eq:radial-geodesic}
 \gamma_\theta(t)=\exp_p(t\theta),
 \qquad \dot\gamma_\theta(t)=\partial_r.
\end{equation}
Let \(P_t:T_pM\to T_{\gamma_\theta(t)}M\) be parallel transport along
\(\gamma_\theta\).  Since \(P_t\theta=\dot\gamma_\theta(t)\), parallel
transport identifies \(\dot\gamma_\theta(t)^\perp\) with the fixed vector
space \(\theta^\perp\).

Define the radial curvature operator
\(R_\theta(t):\theta^\perp\to\theta^\perp\) by
\begin{equation}\label{eq:radial-curvature}
 R_\theta(t)X
 =P_t^{-1}\!\left(
   \Rm_{\gamma_\theta(t)}\bigl(P_tX,\dot\gamma_\theta(t)\bigr)
   \dot\gamma_\theta(t)\right),
 \qquad X\in\theta^\perp.
\end{equation}
For each fixed $\theta$, the right-hand side of
\eqref{eq:radial-curvature} is a composition of smooth objects: the
curvature tensor along \(\gamma_\theta\), the velocity
\(\dot\gamma_\theta\), parallel transport $P_t$, and its inverse.
Consequently, $R_\theta(t)$ is smooth in $t$ for each fixed $\theta$.

The trace of $R_\theta$ is the radial Ricci curvature.  Indeed, choose an
orthonormal basis \(\{e_\alpha\}_{\alpha=1}^{n-1}\) of $\theta^\perp$.
Then
\(\{P_t\theta,P_te_1,\ldots,P_te_{n-1}\}\) is an orthonormal basis of
\(T_{\gamma_\theta(t)}M\), and hence
\begin{equation}\label{eq:radial-curvature-trace}
\begin{aligned}
 \tr R_\theta(t)
 &=\sum_{\alpha=1}^{n-1}
   \langle e_\alpha,R_\theta(t)e_\alpha\rangle_p\\
 &=\sum_{\alpha=1}^{n-1}
   \left\langle P_te_\alpha,
   \Rm(P_te_\alpha,\dot\gamma_\theta)\dot\gamma_\theta
   \right\rangle_{\gamma_\theta(t)}\\
 &=\Ric_{\gamma_\theta(t)}
   \bigl(\dot\gamma_\theta(t),\dot\gamma_\theta(t)\bigr)\\
 &=\Ric(\partial_r,\partial_r)(t,\theta).
\end{aligned}
\end{equation}
Thus \(\Ric(\partial_r,\partial_r)(t,\theta)\) is also smooth in $t$ for
every fixed $\theta$.

We next record the Taylor coefficients of
\eqref{eq:radial-curvature}.  In the parallel frame determined by $P_t$,
differentiation gives
\[
 \frac{\dd}{\dd t}R_\theta(t)
 =P_t^{-1}\!\left(
   (\nabla_{\dot\gamma_\theta}\Rm)_{\gamma_\theta(t)}
   (P_t\,\cdot\,,\dot\gamma_\theta)\dot\gamma_\theta\right).
\]
The derivatives of $P_t$ disappear because the frame is parallel, and
\(\nabla_{\dot\gamma_\theta}\dot\gamma_\theta=0\) because
\(\gamma_\theta\) is a geodesic.  Since $P_0=\Id$, it follows that
\[
 \left.\frac{\dd}{\dd t}R_\theta(t)\right|_{t=0}
 =(\nabla_\theta\Rm)_p(\,\cdot\,,\theta)\theta.
\]
Repeating this argument yields, for every integer $k\geq0$,
\[
 \left.\frac{\dd^k}{\dd t^k}R_\theta(t)\right|_{t=0}
 =(\nabla_\theta^k\Rm)_p(\,\cdot\,,\theta)\theta.
\]
Accordingly, Taylor's theorem gives, for every fixed $N\geq0$,
\begin{equation}\label{eq:radial-curvature-Taylor}
 R_\theta(t)
 =\sum_{k=0}^{N}\frac{t^k}{k!}
   (\nabla_\theta^k\Rm)_p(\,\cdot\,,\theta)\theta
   +O(t^{N+1}).
\end{equation}

Taking traces commutes with differentiation.  Therefore
\[
 \left.\frac{\dd^k}{\dd t^k}\tr R_\theta(t)\right|_{t=0}
 =\tr\!\left((\nabla_\theta^k\Rm)_p
 (\,\cdot\,,\theta)\theta\right).
\]
Consequently, for every fixed integer $\ell\geq0$, the corresponding
radial Ricci expansion is
\begin{equation}\label{eq:radial-Ricci-Taylor}
 \Ric(\partial_r,\partial_r)(t,\theta)
 =\sum_{k=0}^{\ell}\frac{t^k}{k!}
   \bigl(\nabla_\theta^k\Ric\bigr)_p(\theta,\theta)
   +O(t^{\ell+1}).
\end{equation}

\subsection{The Jacobi matrix and geodesic spheres}

For each fixed $\theta$, let
\(J_\theta(t):\theta^\perp\to\theta^\perp\) be the solution of
\begin{equation}\label{eq:Jacobi}
 J_\theta''(t)+R_\theta(t)J_\theta(t)=0,
 \qquad J_\theta(0)=0,
 \qquad J_\theta'(0)=\Id.
\end{equation}
Because $R_\theta(t)$ is smooth in $t$ for fixed $\theta$, standard smooth
dependence for ordinary differential equations shows that $J_\theta(t)$ is
smooth in $t$ as well.  The initial conditions imply
\(J_\theta(t)=t\Id+O(t^3)\).  Hence, for sufficiently small $t>0$,
$J_\theta(t)$ is invertible.

Writing $J_\theta(t)=tK_\theta(t)$, we have $K_\theta(0)=\Id$ and
$K_\theta$ is smooth.  Thus
\begin{equation}\label{eq:Jacobi-inverse}
 J_\theta(t)^{-1}=\frac1tK_\theta(t)^{-1}
\end{equation}
for small $t>0$.  The shape operator of the geodesic sphere
\(\partial B_t(p)\), with respect to its outward unit normal $\partial_r$,
is
\begin{equation}\label{eq:B-def}
 S_\theta(t)=J_\theta'(t)J_\theta(t)^{-1}
 =\frac1t\Id+K_\theta'(t)K_\theta(t)^{-1},
 \qquad
 B_\theta(t):=K_\theta'(t)K_\theta(t)^{-1}.
\end{equation}
In particular, $B_\theta(t)$ is smooth in $t$ for each fixed $\theta$.
Define
\begin{equation}\label{eq:beta-def}
 \beta(t,\theta)=\tr B_\theta(t).
\end{equation}
Then
\begin{equation}\label{eq:Delta-r-beta}
 \Delta r=\tr S_\theta(t)=\frac{n-1}{t}+\beta(t,\theta).
\end{equation}
Differentiating \(S_\theta=J_\theta'J_\theta^{-1}\) and using
\eqref{eq:Jacobi}, we obtain the classical Riccati equation
\begin{equation}\label{eq:S-Riccati}
 S_\theta'(t)+S_\theta(t)^2+R_\theta(t)=0.
\end{equation}
Substitution of \eqref{eq:B-def} into \eqref{eq:S-Riccati} cancels the two
$t^{-2}\Id$ terms and gives
\begin{equation}\label{eq:B-Riccati}
 B_\theta'(t)+\frac2tB_\theta(t)+B_\theta(t)^2+R_\theta(t)=0.
\end{equation}

\subsection{The exact trace identity}

\begin{lemma}\label{lem:trace-Riccati}
Along every radial geodesic,
\begin{equation}\label{eq:beta-ode}
 \beta'(t,\theta)+\frac2t\beta(t,\theta)
 +|B_\theta(t)|^2
 +\Ric(\partial_r,\partial_r)(t,\theta)=0,
\end{equation}
where \(|B_\theta(t)|^2\) denotes
\(\tr\bigl(B_\theta(t)^2\bigr)\).  Consequently,
\begin{equation}\label{eq:beta-integral}
 \beta(t,\theta)
 =-t^{-2}\int_0^t s^2
 \bigl(|B_\theta(s)|^2
 +\Ric(\partial_r,\partial_r)(s,\theta)\bigr)\dd s.
\end{equation}
\end{lemma}

\begin{proof}
By \eqref{eq:beta-def},
\[
 \tr\bigl(B_\theta'(t)\bigr)
 =\frac{\dd}{\dd t}\tr B_\theta(t)=\beta'(t,\theta).
\]
Moreover,
\[
 \tr\bigl(B_\theta(t)^2\bigr)=|B_\theta(t)|^2,
\]
by the notation in the statement, while
\eqref{eq:radial-curvature-trace} gives
\[
 \tr R_\theta(t)=\Ric(\partial_r,\partial_r)(t,\theta).
\]
Taking the trace of \eqref{eq:B-Riccati} therefore proves
\eqref{eq:beta-ode}.  Multiplication by $t^2$ gives
\[
 \frac{\dd}{\dd t}\bigl(t^2\beta(t,\theta)\bigr)
 =-t^2\bigl(|B_\theta(t)|^2
 +\Ric(\partial_r,\partial_r)(t,\theta)\bigr).
\]
Integrating from zero to $t$ yields
\[
 \int_0^t\frac{\dd}{\dd s}\bigl(s^2\beta(s,\theta)\bigr)\dd s
 =-\int_0^t s^2\bigl(|B_\theta(s)|^2
 +\Ric(\partial_r,\partial_r)(s,\theta)\bigr)\dd s.
\]
Because $B_\theta$ is smooth, so is $\beta$, and hence
\(\lim_{s\downarrow0}s^2\beta(s,\theta)=0\).  Dividing the resulting
identity by $t^2$ proves \eqref{eq:beta-integral}.
\end{proof}

\subsection{The critical Jacobi estimate}

\begin{lemma}\label{lem:critical-Jacobi}
Let $n\geq5$ be odd and put
\begin{equation}\label{eq:ell-def}
 \ell=\frac{n-5}{2}.
\end{equation}
Suppose, uniformly in $\theta$, that near \textit{p} 
\begin{equation}\label{eq:radial-assumptions}
 R_\theta(t)=O(t^\ell),
 \qquad
 \Ric(\partial_r,\partial_r)(t,\theta)=O(t^{n-3}).
\end{equation}
Then
\begin{equation}\label{eq:Jacobi-estimates}
 B_\theta(t)=O(t^{\ell+1}),
 \qquad |B_\theta(t)|^2=O(t^{n-3}),
 \qquad \beta(t,\theta)=O(t^{n-2}).
\end{equation}
More precisely, smoothness and \eqref{eq:radial-assumptions} give expansions
\begin{equation}\label{eq:R-leading}
\begin{aligned}
 R_\theta(t)&=t^\ell R_\ell(\theta)+O(t^{\ell+1}),\\
 \Ric(\partial_r,\partial_r)(t,\theta)
 &=t^{n-3}\rho(\theta)+O(t^{n-2}),
\end{aligned}
\end{equation}
and these coefficients satisfy
\begin{equation}\label{eq:B-leading}
 B_\theta(t)=-\frac{t^{\ell+1}}{\ell+3}R_\ell(\theta)
 +O(t^{\ell+2}),
\end{equation}
\begin{equation}\label{eq:beta-leading}
 \beta(t,\theta)
 =-\frac1nA(\theta)t^{n-2}+O(t^{n-1}),
 \qquad
 A(\theta)=\rho(\theta)+\frac{|R_\ell(\theta)|^2}{(\ell+3)^2}.
\end{equation}
Here \(|R_\ell(\theta)|^2=\tr(R_\ell(\theta)^2)\), and $A$ is a smooth
even function on \(\Sph^{n-1}\). 

Moreover 
\begin{equation}\label{eq:Delta-r-critical}
 \Delta r=\frac{n-1}{r}-\frac1nA(\theta)r^{n-2}+O(r^{n-1}),
 \qquad A(-\theta)=A(\theta).
\end{equation}
\end{lemma}

\begin{proof}
We begin with the order of $B_\theta$.  Integrating the Jacobi equation
\eqref{eq:Jacobi} once gives
\[
 J_\theta'(t)=\Id-\int_0^tR_\theta(s)J_\theta(s)\dd s.
\]
The initial conditions imply $J_\theta(s)=O(s)$.  Using
$R_\theta(s)=O(s^\ell)$ in the integral therefore yields
\[
 J_\theta'(t)=\Id+O(t^{\ell+2}),
 \qquad
 J_\theta(t)=t\Id+O(t^{\ell+3}).
\]
It follows that
\[
 J_\theta(t)^{-1}
 =\frac1t\bigl(\Id+O(t^{\ell+2})\bigr)
 =\frac1t\Id+O(t^{\ell+1}).
\]
Consequently,
\[
 S_\theta(t)=J_\theta'(t)J_\theta(t)^{-1}
 =\frac1t\Id+O(t^{\ell+1}),
\]
and hence, by \eqref{eq:B-def},
\begin{equation}\label{eq:B-O}
  B_\theta(t)=O(t^{\ell+1}).  
\end{equation}
Therefore
\[
 |B_\theta(t)|^2=\tr\bigl(B_\theta(t)^2\bigr)
 =O(t^{2\ell+2})=O(t^{n-3}).
\]
Combining this estimate with the radial Ricci estimate in
\eqref{eq:radial-assumptions} and applying the exact identity
\eqref{eq:beta-integral} proves \(\beta(t,\theta)=O(t^{n-2})\).

Because $B_\theta$ is smooth, the first estimate in
\eqref{eq:Jacobi-estimates} allows us to write
\[
 B_\theta(t)=t^{\ell+1}C(\theta)+O(t^{\ell+2})
\]
for a coefficient $C(\theta)$.  Similarly, smoothness gives the two
expansions in \eqref{eq:R-leading}.  Substitution into
\eqref{eq:B-Riccati} and comparison of the coefficient of $t^\ell$ give
\[
 (\ell+1)C(\theta)+2C(\theta)+R_\ell(\theta)=0.
\]
Thus
\[
 C(\theta)=-\frac1{\ell+3}R_\ell(\theta),
\]
which proves \eqref{eq:B-leading}.  Squaring and taking the trace yields
\begin{equation}\label{eq:B-square-leading}
 |B_\theta(t)|^2
 =\frac{t^{2\ell+2}}{(\ell+3)^2}|R_\ell(\theta)|^2
 +O(t^{n-2}).
\end{equation}
Because $2\ell+2=n-3$, equations \eqref{eq:R-leading} and
\eqref{eq:B-square-leading} imply
\begin{equation}\label{eq:trace-source-leading}
 |B_\theta(t)|^2+\Ric(\partial_r,\partial_r)(t,\theta)
 =A(\theta)t^{n-3}+O(t^{n-2}),
\end{equation}
with $A(\theta)$ as in \eqref{eq:beta-leading}.  Substitution into
\eqref{eq:beta-integral} gives
\[
\begin{aligned}
 \beta(t,\theta)
 &=-t^{-2}\int_0^t
   \bigl(A(\theta)s^{n-1}+O(s^n)\bigr)\dd s\\
 &=-\frac1nA(\theta)t^{n-2}+O(t^{n-1}),
\end{aligned}
\]
which proves the remaining expansion.

Extend $\gamma_\theta$ to signed $t$ and use the same parallel frame.  Since
\(\gamma_{-\theta}(t)=\gamma_\theta(-t)\) and the two occurrences of the
radial velocity in \eqref{eq:radial-curvature} cancel its sign, we have
\[
 R_{-\theta}(t)=R_\theta(-t).
\]
Comparing this identity with \eqref{eq:R-leading} gives
\[
 R_\ell(-\theta)=(-1)^\ell R_\ell(\theta).
\]
In the same way, the radial Ricci function satisfies
\(\Ric_{-\theta}(t)=\Ric_\theta(-t)\).  Since $n$ is odd, $n-3$ is even.
By \eqref{eq:R-leading},
\[
 \rho(-\theta)=\rho(\theta).
\]
It follows directly from the definition in \eqref{eq:beta-leading} that
\(A(-\theta)=A(\theta)\).
Moreover, by \eqref{eq:radial-curvature-Taylor},
\[
 R_\ell(\theta)
 =\frac1{\ell!}(\nabla_\theta^\ell\Rm)_p(\,\cdot\,,\theta)\theta,
\]
and by \eqref{eq:radial-Ricci-Taylor},
\[
 \rho(\theta)
 =\frac1{(n-3)!}(\nabla_\theta^{n-3}\Ric)_p(\theta,\theta).
\]
Thus $R_\ell$ and $\rho$ belong to $C^\infty(\Sph^{n-1})$, and hence
$A\in C^\infty(\Sph^{n-1})$.

Combining \eqref{eq:Delta-r-beta} and \eqref{eq:beta-leading}, we obtain
the pole expansion
\begin{equation*}
 \Delta r=\frac{n-1}{r}-\frac1nA(\theta)r^{n-2}+O(r^{n-1}),
 \qquad A(-\theta)=A(\theta).
\end{equation*}
\end{proof}

\begin{lemma}\label{D-r}
Let \(M^3\) be a  three-dimensional Riemannian manifold. Then
\begin{equation}\label{eq;D-r}
  \Delta r
=
\frac{2}{r}
-\frac{1}{3}\Ric_p(\theta,\theta)r
+O(r^2).
\end{equation}

\end{lemma}

\begin{proof}
By \eqref{eq:radial-curvature-Taylor} and \eqref{eq:radial-Ricci-Taylor}, we have
\[
R_\theta(t)=O(1),
\]
and
\[
\Ric(\partial_r,\partial_r)(t,\theta)
=
\Ric_p(\theta,\theta)+O(t).
\]

By \eqref{eq:B-O}, we have
\[
|B_\theta(s)|^2=\tr\bigl(B_\theta(s)^2\bigr)
 =O(s^2).
\]
Substituting these estimates into \eqref{eq:beta-integral}, we obtain
\begin{align*}
\beta(t,\theta)
&=
-t^{-2}
\int_0^t
s^2
\left(
O(s^2)+\Ric_p(\theta,\theta)+O(s)
\right)\,ds \\
&=
-\frac{1}{3}\Ric_p(\theta,\theta)t
+O(t^2).
\end{align*}

Substituting this into \eqref{eq:Delta-r-beta} yields
\[
\Delta r
=
\frac{2}{r}
-\frac{1}{3}\Ric_p(\theta,\theta)r
+O(r^2).
\]
\end{proof}

\begin{lemma}\label{lem:scalar-to-Ricci-jets}
Let $(M,g)$ be a  Riemannian manifold with $\Ric\geq0$  near p, and let $k\geq0$.  If
\begin{equation}\label{eq:scalar-to-Ricci-hypothesis}
 \nabla^j\Scal(p)=0 \qquad (0\leq j\leq k),
\end{equation}
then
\begin{equation}\label{eq:scalar-to-Ricci-conclusion}
 \nabla^j\Ric(p)=0 \qquad (0\leq j\leq k).
\end{equation}
\end{lemma}

\begin{proof}
We argue by induction on the order of differentiation.  Since $\Ric(p)$ is
a nonnegative symmetric bilinear form and
\[
 \tr\Ric(p)=\Scal(p)=0,
\]
we first obtain $\Ric(p)=0$.

Suppose that $1\leq m\leq k$ and that
$\nabla^j\Ric(p)=0$ for $0\leq j<m$.  Fix $\theta,Y\in T_pM$, let
\[
 \gamma_\theta(t)=\exp_p(t\theta),
\]
and let $Y(t)$ be the parallel vector field along $\gamma_\theta$ with
$Y(0)=Y$.  The function
\[
 F(t)=\Ric_{\gamma_\theta(t)}\bigl(Y(t),Y(t)\bigr)
\]
is nonnegative and smooth for all sufficiently small $t$, of either sign.
The Taylor expansion and the induction hypothesis give
\begin{equation}\label{eq:Ricci-directional-Taylor}
 F(t)=\frac{t^m}{m!}
 \bigl(\nabla_\theta^m\Ric\bigr)_p(Y,Y)+O(t^{m+1}).
\end{equation}
If $m$ is odd, then \eqref{eq:Ricci-directional-Taylor} can be written as
\begin{equation}\label{eq:Ricci-directional-Taylor-odd}
 F(t)=\frac{t^m}{m!}
 \left[\bigl(\nabla_\theta^m\Ric\bigr)_p(Y,Y)+O(t)\right].
\end{equation}
Since $F(t)\geq0$ for both signs of $t$, taking $t<0$ and $t>0$ forces
\[
 \bigl(\nabla_\theta^m\Ric\bigr)_p(Y,Y)=0.
\]
By the arbitrariness of $Y$,
\(\bigl(\nabla_\theta^m\Ric\bigr)_p=0\).

If $m$ is even, the inequality $F(t)\geq0$ implies
\[
 \bigl(\nabla_\theta^m\Ric\bigr)_p(Y,Y)\geq0.
\]
Thus \(\bigl(\nabla_\theta^m\Ric\bigr)_p\) is a nonnegative symmetric
bilinear form in $Y$.  Moreover,
\[
 \tr\bigl(\nabla_\theta^m\Ric\bigr)_p
 =\bigl(\nabla_\theta^m\Scal\bigr)(p)=0,
\]
so \(\bigl(\nabla_\theta^m\Ric\bigr)_p=0\).

In either case,
\(\bigl(\nabla_\theta^m\Ric\bigr)_p=0\) for every $\theta\in T_pM$.
The Ricci commutation identities, together with the induction hypothesis,
show that the $m$ covariant-derivative slots of $\nabla^m\Ric$ are
symmetric at $p$.  Polarization in $\theta$ therefore gives
$\nabla^m\Ric(p)=0$.  This completes the induction.
\end{proof}

\begin{remark}\label{rem:odd-order-jets}
The same proof gives the following odd-order consequence on every manifold
with $\Ric\geq0$  near \textit{p}: if
\[
 \left.\nabla^j\Ric\right|_p=0
 \qquad (0\leq j\leq2i),\qquad i\geq0,
\]
then \(\left.\nabla^{2i+1}\Ric\right|_p=0\).  Similarly, if
\[
 \left.\nabla^j\Scal\right|_p=0
 \qquad (0\leq j\leq2i),\qquad i\geq0,
\]
then \(\bigl(\nabla^{2i+1}\Scal\bigr)(p)=0\).
\end{remark}

\begin{proposition}\label{prop:finite-jets-to-radial}
Let $n\geq5$ be odd, and let $(M^n,g)$ be a  Riemannian manifold
with $\Ric\geq0$ near \textit{p}.  Suppose that,
\begin{equation}\label{eq:scalar-jets-radial}
 \nabla^{2i}\Scal(p)=0 \qquad (0\leq 2i\leq n-5),
\end{equation}
and
\begin{equation}\label{eq:Rm-jets-radial}
 \left.\nabla^j\Rm\right|_p=0
 \qquad\left(0\leq j\leq\frac{n-7}{2}\right).
\end{equation}
When $n=5$, condition \eqref{eq:Rm-jets-radial} is empty.  Then the radial
curvature estimates \eqref{eq:radial-assumptions} hold.
\end{proposition}

\begin{proof}
By the scalar-curvature hypothesis,
\[
 \nabla^{2i}\Scal(p)=0 \qquad (0\leq 2i\leq n-5).
\]
Since $n-5$ is even, successive applications of
\remref{rem:odd-order-jets} also give
\[
 \nabla^j\Scal(p)=0 \qquad (0\leq j\leq n-4).
\]
By \lemref{lem:scalar-to-Ricci-jets},
\[
 \left.\nabla^i\Ric\right|_p=0 \qquad (0\leq i\leq n-4).
\]
By \eqref{eq:radial-Ricci-Taylor},
\begin{equation}\label{eq:radial-Ricci-from-jets}
\begin{aligned}
 \Ric(\partial_r,\partial_r)(t,\theta)
 &=\sum_{k=0}^{n-4}\frac{t^k}{k!}
   \bigl(\nabla_\theta^k\Ric\bigr)_p(\theta,\theta)
   +O(t^{n-3})\\
 &=O(t^{n-3}).
\end{aligned}
\end{equation}
Put $\ell=(n-5)/2$.  If $n=5$, then $\ell=0$, and the smoothness of
$R_\theta(t)$ in $t$ gives
\[
 R_\theta(t)=O(1)=O(t^\ell).
\]
If $n\geq7$, by \eqref{eq:radial-curvature-Taylor},
\begin{equation}\label{eq:radial-curvature-from-jets}
 R_\theta(t)
 =\sum_{k=0}^{\ell-1}\frac{t^k}{k!}
   \bigl(\nabla_\theta^k\Rm\bigr)_p(\,\cdot\,,\theta)\theta
   +O(t^\ell)
 =O(t^\ell).
\end{equation}
By \eqref{eq:radial-Ricci-from-jets}, together with
\eqref{eq:radial-curvature-from-jets}, this is precisely
\eqref{eq:radial-assumptions}.
\end{proof}

Let \(\cH_1\) denote the space of degree-one spherical harmonics on
\(\Sph^{n-1}\).  The notation \(a\perp\cH_1\) means orthogonality with
respect to the standard \(L^2(\Sph^{n-1})\) inner product; equivalently,
\begin{equation}\label{eq:H1-orthogonality}
 \int_{\Sph^{n-1}}a(\theta)\theta_i\,\dd\theta=0
 \qquad(1\leq i\leq n).
\end{equation}

\begin{proposition}\label{prop:density-criterion}
Let $n\geq3$,$(M^n,g)$ be a smooth Riemannian manifold
.  Suppose that, in geodesic polar coordinates at $p$, for
some \(a\in C^\infty(\Sph^{n-1})\),
\begin{equation}\label{eq:density-hypothesis}
 \Delta r=\frac{n-1}{r}+a(\theta)r^{n-2}+O(r^{n-1}),
 \qquad a\perp\cH_1.
\end{equation}
Here and below the remainder is uniform in \(\theta\).
Then every Green function with principal singularity $r^{2-n}$ satisfies
\begin{equation}\label{eq:density-conclusion}
G(p,x) = r^{2-n} + H_p + O_1(r).
\end{equation}
Here $H_p$ is a constant.
\end{proposition}

\begin{proof}
For a radial power,
\begin{equation}\label{eq:Delta-r-power}
 \Delta r^{2-n}
 =(2-n)r^{1-n}\left(\Delta r-\frac{n-1}{r}\right).
\end{equation}
Thus \eqref{eq:density-hypothesis} gives
\begin{equation}\label{eq:Delta-pole-error}
 \Delta r^{2-n}=(2-n)a(\theta)r^{-1}+O(1).
\end{equation}
On the unit sphere, the operator
\begin{equation}\label{eq:L-def}
 L=\Delta_{\Sph^{n-1}}+(n-1)
\end{equation}
has kernel \(\cH_1\).  In fact, \(\cH_1\) consists of the restrictions to
\(\Sph^{n-1}\subset\R^n\) of the coordinate functions on \(\R^n\).  These
are precisely the first eigenfunctions of \(-\Delta_{\Sph^{n-1}}\), with
eigenvalue \(n-1\).  Hence \(\ker L=\cH_1\).  Since \(L\) is
self-adjoint elliptic on the compact
sphere, the Fredholm alternative shows that \(L\phi=(n-2)a\) is solvable
precisely when its right-hand side is orthogonal to \(\ker L\).  The
assumption \(a\perp\cH_1\) therefore permits a solution \(\phi\),
unique after imposing \(\phi\perp\cH_1\), of
\stepcounter{equation}
\begin{equation}\label{eq:phi-equation}
 L\phi=(n-2)a.
\end{equation}
Here \(\phi=\phi(\theta)\in C^\infty(\Sph^{n-1})\); in particular,
\(\phi\) depends only on the angular variable \(\theta\).

By the smoothness of the metric, in geodesic polar coordinates,
\begin{equation}\label{eq:polar-metric}
 g=\dd r^2+r^2h_r,
 \qquad
 h_r=g_{\Sph^{n-1}}+O_{C^1}(r^2).
\end{equation}
Here $O_{C^1}(r^2)$ means that both $h_r-g_{\Sph^{n-1}}$ and its first
derivatives with respect to the spherical coordinates $\theta$ are
$O(r^2)$.
The $C^1$ control is important because the Laplace--Beltrami
operator differentiates the metric coefficients once.  From its coordinate
formula,
\begin{align*}
 \Delta_{h_r}\phi
 &=\frac1{\sqrt{\det h_r}}
 \partial_a\bigl(\sqrt{\det h_r}\,h_r^{ab}\partial_b\phi\bigr)\\
 &=\frac1{\sqrt{\det g_{\Sph^{n-1}}}}\bigl(1+O(r^2)\bigr)
 \left[
 \partial_a\bigl(\sqrt{\det g_{\Sph^{n-1}}}\,
 g_{\Sph^{n-1}}^{ab}\partial_b\phi\bigr)
 +O(r^2)
 \right]\\
 &=\frac1{\sqrt{\det g_{\Sph^{n-1}}}}
 \partial_a\bigl(\sqrt{\det g_{\Sph^{n-1}}}\,
 g_{\Sph^{n-1}}^{ab}\partial_b\phi\bigr)
 +O(r^2).
\end{align*}
\eqref{eq:polar-metric} gives
\begin{equation}\label{eq:angular-laplacian-expansion}
 \Delta_{h_r}\phi=\Delta_{\Sph^{n-1}}\phi+O(r^2).
\end{equation}
For a general function $u(r,\theta)$,
\[
 \Delta_gu=\partial_r^2u+(\Delta r)\partial_ru
 +r^{-2}\Delta_{h_r}u.
\]
Applying this to $u=r\phi(\theta)$, and using
\(\Delta r=(n-1)/r+O(r)\) together with
\eqref{eq:angular-laplacian-expansion}, yields the sharper estimate
\begin{equation}\label{eq:Delta-rphi}
 \Delta_g(r\phi)=r^{-1}L\phi+O(r),
\end{equation}
Equations \eqref{eq:Delta-pole-error}, \eqref{eq:phi-equation}, and
\eqref{eq:Delta-rphi} show that
\[
 F:=\Delta_g(r^{2-n}+r\phi)
\]
is bounded in a punctured normal ball.  Extend $F$ arbitrarily at $p$,
choose a small normal ball $B_\varepsilon(p)$, and solve the weak Dirichlet
problem
\[
 \Delta z=-F\quad\text{in }B_\varepsilon(p),
 \qquad z=0\quad\text{on }\partial B_\varepsilon(p).
\]
For every finite $s$, elliptic estimates give $z\in W^{2,s}$; see
\cite[Chapters~8--9]{GilbargTrudinger}.  Taking $s>n$ and applying Sobolev
embedding yields $z\in C^{1,\alpha}$ for some \(\alpha>0\).

The function \(r\phi\in W^{1,\infty}_{\mathrm{loc}}\): indeed
\(|r\phi|\leq Cr\), and its radial and angular gradient components are
uniformly bounded.  Its normal flux through \(\partial B_\delta(p)\) is
therefore $O(\delta^{n-1})$, so it carries no point mass at $p$.  The
same is true of $z$, since $z\in C^1$.  Integrating by parts on
\(B_\varepsilon\setminus B_\delta\) and sending \(\delta\downarrow0\), the
normalized flux of $r^{2-n}$ gives
\[
 -\Delta\Phi=(n-2)|\Sph^{n-1}|\delta_p,
 \qquad
 \Phi:=r^{2-n}+r\phi+z.
\]
Thus $\Phi$ is a local fundamental solution with normalization
\eqref{eq:green-normalization}.  Moreover,
\begin{equation}\label{eq:Phi-expansion}
 \Phi=r^{2-n}+z(p)+O_1(r),
\end{equation}
because $r\phi=O_1(r)$ and
\(z-z(p)=O_1(r)\).  Notice that $r\phi$ need not be $C^1$ at $p$;
the notation $O_1(r)$ only asserts the punctured-ball estimates in its
definition in Section~\ref{sec:colding}.

If $G$ is any Green function with the same normalization, set
\[
 \mu:=G-\Phi.
\]
The function \(\mu\) is distributionally harmonic in the whole ball
\(B_\varepsilon(p)\).  By Weyl's lemma \cite{Jost},
\(\mu\in C^\infty(B_\varepsilon(p))\).  Consequently,
\[
 \mu=\mu(p)+O_1(r).
\]
Combining this with \eqref{eq:Phi-expansion}, we obtain
\begin{align*}
 G
 &=\Phi+\mu\\
 &=r^{2-n}+z(p)+\mu(p)+O_1(r)\\
 &=r^{2-n}+H_p+O_1(r),
 \qquad H_p:=z(p)+\mu(p),
\end{align*}
which proves \eqref{eq:density-conclusion}.
\end{proof}

\begin{proof}[Proof of \thmref{thm:pure-pole}]
For $n\geq5$, the local argument in
\propref{prop:finite-jets-to-radial} gives
\eqref{eq:radial-assumptions}, and \lemref{lem:critical-Jacobi} gives
\eqref{eq:Delta-r-critical}. $A(\theta)$ is smooth and even on
$\Sph^{n-1}$, and hence orthogonal to $\cH_1$. For $n=3$,
\lemref{D-r} gives \eqref{eq;D-r}. Since $\Ric_p(\theta,\theta)$ is smooth
and even on $\Sph^{n-1}$, $\Ric_p(\theta,\theta)$ is orthogonal to
$\cH_1$. In both cases, \propref{prop:density-criterion} gives the
expansion
\[
 G(p,x)=r^{2-n}+H_p+O_1(r)
\]
then
\[
 |\nabla b|^2
 =1-\frac{2(n-1)}{n-2}H_pr^{n-2}+O(r^{n-1}),
\]
and hence
\[
 |\nabla b|^2-1
 =-\frac{2(n-1)}{n-2}H_pr^{n-2}+O(r^{n-1}).
\]
 When $\mathrm{Ric} \ge 0 $ and
$G$ is the normalized minimal positive Green function,
\propref{prop:Colding}(iv) implies that $H_p\geq0$.
\end{proof}

\section{Proof of Theorem 1.3}\label{sec:global}

\begin{lemma}\label{lem:gradb-expansion}
Let $n\geq3$ be odd ,$(M^{n},g)$ be a complete nonparabolic Riemannian manifold with nonnegative Ricci curvature
.  Suppose that \(|\nabla b|^2\) extends across the
pole $p$ as a $C^{n-2}$ function and that
\begin{equation}\label{eq:G-H-O1}
 G(p,x)=r^{2-n}+H+O_1(r).
\end{equation}
Then
\begin{equation}\label{eq:gradb-strong-decay}
 H=0,
 \qquad |\nabla b|^2=1+O(r^{n-1}),
\end{equation}
and
\begin{equation}\label{eq:small-defect-limit}
 \lim_{s\downarrow0}s^{3-n}A'(s)=0.
\end{equation}
\end{lemma}

\begin{proof}
Since $b=G^{1/(2-n)}$, expansion \eqref{eq:G-H-O1} gives
\begin{equation}\label{eq:b-expansion}
 b=r-\frac{H}{n-2}r^{n-1}+O_1(r^n).
\end{equation}
Consequently,
\begin{equation}\label{eq:gradb-expansion}
 |\nabla b|^2
 =1-\frac{2(n-1)}{n-2}Hr^{n-2}+O(r^{n-1}).
\end{equation}

Fix a unit vector $\theta\in T_pM$ and define
\[
 f_\theta(t)=|\nabla b|^2(\exp_p(t\theta))-1.
\]
Because \(|\nabla b|^2\in C^{n-2}\), the function $f_\theta$ is $C^{n-2}$
in $t$.  By \eqref{eq:gradb-expansion},
\begin{equation}\label{eq:signed-geodesic-expansion}
 f_\theta(t)=-\frac{2(n-1)}{n-2}H|t|^{n-2}+O(|t|^{n-1}).
\end{equation}
On the other hand, the Taylor expansion of $f_\theta$ has the form
\begin{equation}\label{eq:signed-geodesic-Taylor}
 f_\theta(t)=a_\theta t^{n-2}+o(|t|^{n-2}).
\end{equation}
Indeed, the estimate $f_\theta(t)=O(|t|^{n-2})$ forces every Taylor coefficient
of degree below $n-2$ to vanish.  Since $n-2$ is odd, comparison of
\eqref{eq:signed-geodesic-expansion} and
\eqref{eq:signed-geodesic-Taylor}, first for $t>0$ and then for $t<0$,
gives
\[
 -\frac{2(n-1)}{n-2}H=0.
\]
Thus $H=0$, and \eqref{eq:gradb-expansion} yields
\[
 |\nabla b|^2=1+O(r^{n-1}).
\]

We now prove that
\[
\lim_{s\to 0}s^{3-n}A'(s)=0.
\]

By \eqref{eq:third-monotonicity}, the function
\[
s^{3-n}A'(s)
\]
is monotonically increasing. On the other hand, by
\propref{prop:Colding}(ii),
\(A(s)\) is monotonically decreasing. Hence
\[
s^{3-n}A'(s)\le 0.
\]

Suppose, to the contrary, that
\[
\lim_{s\to 0}s^{3-n}A'(s)\neq 0.
\]
Then there exist \(s_0>0\) and \(C>0\) such that
\[
s_0^{3-n}A'(s_0)=-C<0.
\]
Since \(s^{3-n}A'(s)\) is monotonically increasing, for every
\(0<t\le s_0\) we have
\[
t^{3-n}A'(t)\le -C.
\]
Therefore,
\[
A'(t)\le -Ct^{n-3}.
\]

By \propref{prop:Colding}(v),
\[
\lim_{t\to 0}A(t)=|S^{n-1}|.
\]
It follows that
\[
A(s)-|S^{n-1}|
=
\int_0^s A'(t)\,dt
\le
-C\int_0^s t^{n-3}\,dt
=
-\frac{C}{n-2}s^{n-2}
<0.
\]
Hence
\begin{equation}\label{eq:ASC}
  \left|A(s)-|S^{n-1}|\right|
\ge
\frac{C}{n-2}s^{n-2}.
\end{equation}

On the other hand, by \eqref{eq:I1} and \eqref{eq:A-def}, and
$b/r\to1$ as $r\to0$,
\begin{align*}
A(s)-|S^{n-1}|
&=
I_{|\nabla b|^2}(s)-I_1(s)\\
&=
s^{1-n}\int_{b=s}
\bigl(|\nabla b|^2-1\bigr)|\nabla b|\,dA\\
&=
O(s^{n-1})\,s^{1-n}
\int_{b=s}|\nabla b|\,dA\\
&=
O(s^{n-1})\,|S^{n-1}|\\
&=
O(s^{n-1}).
\end{align*}
This contradicts \eqref{eq:ASC}. Therefore,
\[
  \lim_{s\to 0}s^{3-n}A'(s)=0.
\]
\end{proof}

\begin{proof}[Proof of \thmref{thm:odd-rigidity}]
\thmref{thm:pure-pole} gives
\[
 G=r^{2-n}+H_p+O_1(r).
\]
\lemref{lem:gradb-expansion} gives
\(\lim_{s\downarrow0}s^{3-n}A'(s)=0\).  \lemref{lem:zero-defect}
then yields \((M^n,g)\cong\R^n\).
\end{proof}

\begin{corollary}\label{cor:dimension-three}
Let $(M^3,g)$ be a complete Riemannian manifold, nonparabolic, with
$\Ric\geq0$.  If
\[
 |\nabla b|^2\in C^1\text{ near }p,
\]
then $(M^3,g)$ is isometric to $\R^3$.
\end{corollary}

\begin{proof}
This follows directly from \thmref{thm:odd-rigidity}.
\end{proof}

\begin{corollary}\label{cor:dimension-five}
Let $(M^5,g)$ be complete and nonparabolic with $\Ric\geq0$.  If, at the
Green pole $p$,
\[
 \Scal(p)=0
 \qquad\text{and}\qquad
 |\nabla b|^2\in C^3\text{ near }p,
\]
then $(M^5,g)$ is isometric to $\R^5$.
\end{corollary}

\begin{proof}
All the conditions in \eqref{eq:curvature-jet-conditions} hold for $n=5$.  \thmref{thm:odd-rigidity}
therefore yields $(M^5,g)\cong\R^5$.
\end{proof}

\begin{corollary}\label{cor:dimension-seven}
Let $(M^7,g)$ be complete and nonparabolic with $\Ric\geq0$.  Suppose that
\[
 \Rm(p)=0,
 \qquad \Delta\Scal(p)=0,
 \qquad |\nabla b|^2\in C^5\text{ near }p.
\]
Then $(M^7,g)$ is isometric to $\R^7$.
\end{corollary}

\begin{proof}
The argument in the proof of \lemref{lem:scalar-to-Ricci-jets} gives
\[
 \nabla^2\Ric|_p\geq0.
\]
On the other hand,
\[
 0=\Delta\Scal(p)
 =\tr(\nabla^2\Scal)(p)
 =\left.\tr\bigl(\tr\nabla^2\Ric\bigr)\right|_p.
\]
Therefore
\[
 \left.\nabla^2\Ric\right|_p=0.
\]
Consequently, all the conditions in
\eqref{eq:curvature-jet-conditions} hold for $n=7$.
\thmref{thm:odd-rigidity} now yields $(M^7,g)\cong\R^7$.
\end{proof}

\section{Sharpness of the regularity condition and even‑dimensional obstructions}\label{sec:even}

\begin{proposition}
\label{prop:even-counterexamples}
For every $n\geq 3$, there is a complete nonflat rotationally symmetric
metric on $\R^n$ with \(\sec\geq0\) and positive asymptotic volume ratio
such that the metric is Euclidean near its pole $p$ and
\begin{equation}\label{eq:even-pure-pole}
 G(p,x)=r^{2-n}+H,
 \qquad H>0,
\end{equation}
\begin{equation}\label{eq:gradb-even}
 |\nabla b|^2=(1+Hr^{n-2})^{-2\frac{n-1}{n-2}}
\end{equation}
near $p$.
\end{proposition}

\begin{proof}
Fix $a\in(0,1)$.  Choose a smooth nonincreasing function
\(\varphi:[0,\infty)\to[a,1]\) such that \(\varphi=1\) on \([0,1]\),
\(\varphi=a\) on \([2,\infty)\), and \(\varphi'<0\) somewhere.  Set
\[
 f(r)=\int_0^r\varphi(s)\dd s,
 \qquad
 g=\dd r^2+f(r)^2g_{\Sph^{n-1}}.
\]
Then $f=r$ near zero, $f''\leq0$, and $0<a\leq f'\leq1$.
Formula \eqref{eq:model-curvatures} shows that \(\sec\geq0\), and the metric
is nonflat because $f''<0$ somewhere.  It is complete, and since
 \(f(r)=ar+c\) for large $r$, it has positive asymptotic volume ratio
\(a^{n-1}\).

The normalized minimal positive Green function is
\[
 G(r)=(n-2)\int_r^\infty f(s)^{1-n}\dd s.
\]
For $0<r\leq1$, splitting the integral at $1$ gives
\[
 G(r)=r^{2-n}+H,
 \qquad
 H=(n-2)\int_1^\infty f(s)^{1-n}\dd s-1.
\]
Because $f(s)\leq s$ everywhere and $f(s)<s$ somewhere,
\[
 (n-2)\int_1^\infty f(s)^{1-n}\dd s
 >(n-2)\int_1^\infty s^{1-n}\dd s=1,
\]
so $H>0$.  Direct calculation gives
\begin{equation*}
 |\nabla b|^2=(1+Hr^{n-2})^{-2\frac{n-1}{n-2}}
\end{equation*}
near $p$.  
\end{proof}

\begin{corollary}[Sharpness]\label{cor:odd-sharpness}
For every odd integer \(n\ge 3\), there exists a complete, nonflat smooth Riemannian metric \(g\) on \(\mathbb{R}^n\) satisfying
\[
\sec \ge 0.
\]
Moreover, \(g\) is Euclidean in a neighborhood of the pole \(p\), and
\[
|\nabla b|^2 \in C^{n-3,1}\setminus C^{n-2},
\]
while \((\mathbb{R}^n,g)\) is not isometric to Euclidean space.
\end{corollary}

\begin{proof}
Let \(g\) be the smooth metric on \(\mathbb{R}^n\) constructed in
\propref{prop:even-counterexamples}.
By \propref{prop:even-counterexamples}, the metric \(g\) is complete, satisfies
\[
\sec \ge 0,
\]
is Euclidean in a neighborhood of the pole \(p\), and
\((\mathbb{R}^n,g)\) is not isometric to Euclidean space. Moreover,
\[
|\nabla b|^2
=
\left(1+Hr^{n-2}\right)^{-\frac{2(n-1)}{n-2}},
\qquad H>0.
\]

Since \(n-2\) is odd, in normal coordinates centered at \(p\) we have
\[
r^{n-2}=|x|^{n-2}.
\]
The classical regularity of \(|x|^{n-2}\) is
\[
|x|^{n-2}\in C^{n-3,1}\setminus C^{n-2}.
\]
The function
\[
z\longmapsto (1+Hz)^{-\frac{2(n-1)}{n-2}}
\]
is smooth in a neighborhood of \(z=0\). Composing with this smooth function
preserves the critical regularity. Hence
\[
|\nabla b|^2\in C^{n-3,1}.
\]

By expanding \eqref{eq:gradb-even} near $0$,
\[
 |\nabla b|^2-1
 =-\frac{2(n-1)}{n-2}H|x|^{n-2}+O\bigl(|x|^{2n-4}\bigr).
\]
Therefore $|\nabla b|^2\in C^{n-3,1}\setminus C^{n-2}$.
\end{proof}

\begin{remark}
This corollary shows that the condition
\[
|\nabla b|^2\in C^{n-2}
\]
in \thmref{thm:odd-rigidity} is sharp: merely assuming
\[
|\nabla b|^2\in C^{n-3,1}
\]
is not enough to force the manifold to be Euclidean, even if the metric is
Euclidean in a neighborhood of the pole.
\end{remark}

Suppose, even in an even dimension, that the favorable local expansion
\[
 G=r^{2-n}+H_p+O_1(r)
\]
has already been established.  The calculation in the proof of
\lemref{lem:gradb-expansion} still gives
\[
 |\nabla b|^2
 =1-\frac{2(n-1)}{n-2}H_pr^{n-2}+O(r^{n-1}).
\]
If $n$ is even, however,
\[
 r^{n-2}=|x|^{n-2}
 =\bigl((x^1)^2+\cdots+(x^n)^2\bigr)^{(n-2)/2}
\]
is a smooth polynomial in normal coordinates. Since $n-2$ is even,
$H_pr^{n-2}$ is smooth, so the regularity of $|\nabla b|^2$ cannot force
$H_p=0$.
The following construction shows that this is a genuine geometric
obstruction rather than a limitation of the proof.

\begin{corollary}
For every even integer \(n\ge 4\), there exists a complete, nonflat smooth
Riemannian metric \(g\) on \(\mathbb{R}^n\) satisfying
$
\sec \ge 0.
$
Moreover, \(g\) is Euclidean in a neighborhood of the pole \(p\), and
$
|\nabla b|^2
$
admits a smooth extension across the pole \(p\).
\end{corollary}

\begin{proof}
Let \(g\) be the smooth metric on \(\mathbb{R}^n\) constructed in
\propref{prop:even-counterexamples}.
By \propref{prop:even-counterexamples}, the metric \(g\) is complete, satisfies
\[
\sec \ge 0,
\]
is Euclidean in a neighborhood of the pole \(p\), and
\((\mathbb{R}^n,g)\) is not isometric to Euclidean space. Moreover,
\[
|\nabla b|^2
=
\left(1+Hr^{n-2}\right)^{-\frac{2(n-1)}{n-2}},
\qquad H>0.
\]

Since \(n\ge 4\) is even, in normal coordinates centered at \(p\),
\[
r^{n-2}=|x|^{n-2}
\]
is a polynomial. Near the pole,
\[
1+Hr^{n-2}>0.
\]
Therefore the function
\[
z\longmapsto (1+Hz)^{-\frac{2(n-1)}{n-2}}
\]
is real analytic near \(z=0\). Since \(r^{n-2}\) is a smooth polynomial in
the normal coordinates, it follows that
$
|\nabla b|^2
$
extends smoothly across the pole \(p\).
\end{proof}


\begin{thebibliography}{99}

\bibitem{Aubin}
T.~Aubin,
\emph{Some Nonlinear Problems in Riemannian Geometry},
Springer Monographs in Mathematics, Springer-Verlag, Berlin, 1998.

\bibitem{Colding}
T.~H. Colding,
New monotonicity formulas for Ricci curvature and applications. I,
\emph{Acta Math.} \textbf{209} (2012), no.~2, 229--263,
\href{https://doi.org/10.1007/s11511-012-0086-2}
{doi:10.1007/s11511-012-0086-2}.

\bibitem{GilbargSerrin}
D.~Gilbarg and J.~Serrin,
On isolated singularities of solutions of second order elliptic differential
equations,
\emph{J. Analyse Math.} \textbf{4} (1955/56), 309--340,
\href{https://doi.org/10.1007/BF02787726}{doi:10.1007/BF02787726}.

\bibitem{GilbargTrudinger}
D.~Gilbarg and N.~S. Trudinger,
\emph{Elliptic Partial Differential Equations of Second Order},
Classics in Mathematics, reprint of the 1998 edition,
Springer-Verlag, Berlin, 2001.

\bibitem{Grigoryan}
A.~Grigor'yan,
\emph{Heat Kernel and Analysis on Manifolds},
AMS/IP Studies in Advanced Mathematics, vol.~47,
American Mathematical Society, Providence, RI, 2009.

\bibitem{Jost}
J.~Jost,
\emph{Riemannian Geometry and Geometric Analysis}, sixth ed.,
Universitext, Springer, Heidelberg, 2011.

\bibitem{LiTam}
P.~Li and L.-F.~Tam,
Green's functions, harmonic functions, and volume comparison,
\emph{J. Differential Geom.} \textbf{41} (1995), no.~2, 277--318.

\bibitem{Petersen}
P.~Petersen,
\emph{Riemannian Geometry}, third ed.,
Graduate Texts in Mathematics, vol.~171, Springer, Cham, 2016.

\bibitem{Varopoulos}
N.~T. Varopoulos,
The Poisson kernel on positively curved manifolds,
\emph{J. Funct. Anal.} \textbf{44} (1981), 359--380.

\end{thebibliography}
\end{document}